\documentclass[a4paper,british]{amsart}

\usepackage{amsthm}
\usepackage{prettyref}
\usepackage[colorlinks]{hyperref}
\usepackage{mathtools}

\newrefformat{sub}{Section~\ref{#1}}
\newrefformat{thm}{Theorem~\ref{#1}}
\newrefformat{lem}{Lemma~\ref{#1}}
\newrefformat{prop}{Proposition~\ref{#1}}
\newrefformat{conj}{Conjecture~\ref{#1}}
\newrefformat{def}{Definition~\ref{#1}}
\newrefformat{fact}{Fact~\ref{#1}}
\newrefformat{cor}{Corollary~\ref{#1}}
\newrefformat{rem}{Remark~\ref{#1}}

\theoremstyle{plain}
\newtheorem{thm}{Theorem}
\newtheorem{prop}[thm]{Proposition}
\newtheorem{cor}[thm]{Corollary}
\newtheorem{lem}[thm]{Lemma}
\newtheorem{conj}[thm]{Conjecture}
\makeatletter
\newcommand\nt@name{Theorem}
\newtheorem*{nt@thm}{\nt@name}
\newenvironment{namedthm}[1]
{\renewcommand\nt@name{#1}
  \begin{nt@thm}}
  {\end{nt@thm}}
\makeatother

\theoremstyle{definition}
\newtheorem{defn}[thm]{Definition}
\newtheorem{nota}[thm]{Notation}

\theoremstyle{remark}
\newtheorem{rem}[thm]{Remark}

\numberwithin{thm}{section}

\newcommand{\ot}{\mathrm{ot}}
\newcommand{\vj}{v_J}
\newcommand{\on}{\mathbf{On}}
\newcommand{\crit}{\mathrm{crit}}

\title{Factorisation of germ-like series}

\author{Sonia L'Innocente}
\address[S.\ L'Innocente]{School of Science and Technology, Mathematics Division, University of Camerino, Via Madonna delle Carceri 9, I-62032 Camerino, Italy}
\email{sonia.linnocente@unicam.it}

\author{Vincenzo Mantova}
\address[V.\ Mantova]{Scuola Normale Superiore, Faculty of Sciences, Piazza dei Cavalieri 7, I-56126 Pisa, Italy}
\curraddr{School of Mathematics, University of Leeds, Leeds LS2 9JT, United Kingdom}
\email{v.l.mantova@leeds.ac.uk}

\thanks{Both authors gratefully acknowledge the support of the project
  FIRB2010 ``New advances in the Model theory of exponentiation''
  RBFR10V792, of which the first author is the principal
  investigator. The second author was also supported by the ERC-AdG
  ``Diophantine Problems'' 267273.}

\date{30th January 2017}

\subjclass[2010]{Primary 13F25, secondary 03E10, 12J15, 06F25}
\keywords{Generalized power series, unique factorisation, order-value}

\begin{document}

\begin{abstract}
  A classical tool in the study of real closed fields are the fields
  $K((G))$ of generalised power series (i.e., formal sums with
  well-ordered support) with coefficients in a field $K$ of
  characteristic 0 and exponents in an ordered abelian group $G$. A
  fundamental result of Berarducci ensures the existence of
  irreducible series in the subring $K((G^{\leq 0}))$ of $K((G))$
  consisting of the generalised power series with non-positive
  exponents.

  It is an open question whether the factorisations of a series in
  such subring have common refinements, and whether the factorisation
  becomes unique after taking the quotient by the ideal generated by
  the non-constant monomials. In this paper, we provide a new class of
  irreducibles and prove some further cases of uniqueness of the
  factorisation.
\end{abstract}

\maketitle

\section{Introduction}

If $K$ is a field and $G$ an additive abelian ordered group, a
\emph{formal series} with \emph{coefficients} in $K$ and
\emph{exponents} in $G$ is a formal sum
$a = \sum_\gamma a_\gamma t^\gamma$, where $a_\gamma \in K$ and
$\gamma \in G$. We call \emph{support} of $a$ the set
$S_a := \{\gamma \in G : a_\gamma \neq 0\}$.  A formal series $a$ is
said to be a \emph{generalised power series} if its support $S_a$ is
well-ordered.  The collection of all generalised power series, denoted
by $K((G))$, is a field with respect to the obvious operations $+$ and
$\cdot$ defined for ordinary power series (see~\cite{Hahn1995}).

When $K$ is ordered, then $K((G))$ has a natural order as well,
obtained by stipulating that $0 < t^\gamma < a$ for any
$\gamma \in G^{>0}$ and for any positive element $a$ of the field
$K$. Moreover, if $K$ is real closed and $G$ is divisible, then
$K((G))$ is real closed. Conversely, any ordered field can be
represented as a subfield of some $\mathbb{R}((G))$
\cite{Gleyzal1937}.

For these reasons, the field $K((G))$ is a valuable tool for the study
of real closed fields. One can use them to prove, for instance, that
every real closed field $R$ has an integer part (i.e., a subring $Z$
such that for all $x \in R$ there exists a unique integer part
$\lfloor x \rfloor \in Z$ of $x$ such that
$\lfloor x \rfloor \leq x < \lfloor x \rfloor +
1$)~\cite{Mourgues1993}. For example, $\mathbb{Z} + K((G^{< 0}))$ is
an integer part of $K((G))$, where $K((G^{< 0}))$ is the subring of
the series with the support contained in the negative part $G^{< 0}$
of the group $G$.

The ring $\mathbb{Z} + K((G^{< 0}))$ has a non-trivial arithmetic
behaviour, some of which is already visible in
$K + K((G^{< 0})) = K((G^{\leq 0}))$. When $G$ is divisible, the ring
$K((G^{\leq 0}))$ is non-noetherian, as for instance we have
$t^{-1} = t^{-\frac{1}{2}} \cdot t^{-\frac{1}{2}} = t^{-\frac{1}{4}}
\cdot t^{-\frac{1}{4}} \cdot t^{-\frac{1}{4}} \cdot t^{-\frac{1}{4}} =
\ldots$.  However, Berarducci~\cite{Berarducci2000} proved that
$K((G^{\leq 0}))$, when $\mathbb{Q} \subseteq G$, contains irreducible
series, such as $1 + \sum_{n} t^{-\frac{1}{n}}$, answering a question
of Conway \cite{Conway1976}; in fact, his result implies that
$1 + \sum_{n} t^{-\frac{1}{n}}$ is irreducible in the ring of omnific
integers, the natural integer part of surreal numbers, which are also
of the form $\mathbb{Z} + \mathbb{R}((G^{< 0}))$ for a suitable $G$.

In order to state Berarducci's result, let the \emph{order type}
$\ot(a)$ of a power series $a \in K((G^{\leq 0}))$ be the ordinal
number representing the order type of its support $S_a$. Moreover, let
$J$ be the ideal of the series that are divisible by $t^\gamma$ for
some $\gamma \in G^{<0}$ (as noted before for $\gamma = -1$, such
series cannot be factored into irreducibles when $G$ is divisible,
since
$t^\gamma = t^{\frac{\gamma}{2}} t^{\frac{\gamma}{2}} = \ldots$).

\begin{thm}[{\cite[Thm.\ 10.5]{Berarducci2000}}]\label{maintB}
  If $a \in K((\mathbb{R}^{\leq 0})) \setminus J$ (equivalently,
  $a \in K((\mathbb{R}^{\leq 0}))$ not divisible by $t^\gamma$ for any
  $\gamma < 0$) has order type $\omega^{\omega^\alpha}$ for some
  ordinal $\alpha$, then both $a$ and $a + 1$ are irreducible.
\end{thm}

This result was obtained by constructing a function resembling a
valuation but taking values into ordinal numbers.

\begin{defn}[{\cite[Def.\ 5.2]{Berarducci2000}}]
  For $a \in K((G^{\leq 0}))$, the \emph{order-value} $\vj(a)$ of $a$
  is:
  \begin{enumerate}
  \item if $a \in J$, then $\vj(a):=0$;
  \item if $a \in J+K$ and $a \notin J$, then $\vj(a):=1$;
  \item if $a \notin J+K$, then
    $\vj(a) := \min \{\ot(a') \,:\, a-a' \in J+K\}$.
  \end{enumerate}
\end{defn}

The difficult key result of~\cite{Berarducci2000} is that for
$G=\mathbb{R}$ the function $\vj$ is \emph{multiplicative}.

\begin{thm}[{\cite[Thm.\ 9.7]{Berarducci2000}}]
  For all $a,b\in K((\mathbb{R}^{\leq 0}))$ we have
  $\vj(ab) = \vj(a) \odot \vj(b)$ (where $\odot$ is Hessenberg's
  natural product on ordinal numbers).
\end{thm}

This immediately implies, for instance, that the ideal $J$ is prime,
so the quotient ring of \emph{germs} $K((\mathbb{R}^{\leq 0})) / J$ is
an integral domain, and also each elements admits a factorisation into
irreducibles (in fact, $J$ is prime for arbitrary choices of $G$,
see~\cite{Pitteloud2002}; for further extensions to arbitrary groups
$G$, see~\cite{Biljakovic2006}).

The above comments and theorems support and motivate the following
conjectures. If $a = b_1 \cdot \ldots \cdot b_n$ is a factorisation of
a series $a$, possibly with some reducible factors, a
\emph{refinement} is another factorisation of $a$ obtained by
replacing each $b_i$ with a further factorisation of $b_i$. More
formally, a refinement is a factorisation
$a = c_1 \cdot \ldots \cdot c_m$ such that, up to reordering
$c_{1},\dots,c_{m}$,
$b_i = k_i \cdot c_{m_i+1} \cdot \ldots \cdot c_{m_{i+1}}$ for some
constants $k_i \in K^*$ and some natural numbers
$0 = m_1 \leq \dots \leq m_{n+1} = m$.

\begin{conj}[Conway~\cite{Conway1976}]
  For every non-zero series $a \in K((\mathbb{R}^{\leq 0}))$, any two
  factorisations of $a$ admit common refinements.
\end{conj}

For instance, it is easy to verify that for all $\gamma < 0$, any two
factorisations of $t^\gamma$ admit a common refinement. Similarly, any
polynomial in $t^{-\gamma}$ with coefficients in $K$ has infinitely
many factorisations, but again any two of them admit a common
refinement.

\begin{conj}[Berarducci~\cite{Berarducci2000}]
  Every non-zero germ in $K((\mathbb{R}^{\leq 0})) / J$ admits a
  unique factorisation into irreducibles.
\end{conj}

Berarducci's work was partially strengthened by
Pitteloud~\cite{Pitteloud2001}, who proved that any (irreducible)
series in $K((\mathbb{R}^{\leq 0}))$ of order type $\omega$ or
$\omega + 1$ and order-value $\omega$ is actually prime.

Adapting Pitteloud's technique, we shall prove that the germs of
order-value $\omega$ are prime in $K((\mathbb{R}^{\leq 0})) / J$; in
particular, the germs of order-value at most $\omega^3$ admit a unique
factorisation into irreducibles, supporting Berarducci's conjecture.

\begin{namedthm}{Theorems~\ref{thm:prime-in-j}-\ref{thm:ufdq}}
  All germs in $K((\mathbb{R}^{\leq 0})) / J$ of order-value $\omega$
  are prime. Every non-zero germ in $K((\mathbb{R}^{\leq 0})) / J$ of
  order-value $\leq \omega^3$ admits a unique factorisation into
  irreducibles.
\end{namedthm}

Moreover, we shall isolate the notion of \emph{germ-like} series: we
say that $a \in K((\mathbb{R}^{\leq 0}))$ is germ-like if either
$\ot(a) = \vj(a)$, or $\vj(a) > 1$ and $\ot(a) = \vj(a) + 1$ (see
\prettyref{def:germ-like}). The main result of~\cite{Berarducci2000}
can be rephrased as saying that germ-like series of order-value
$\omega^{\omega^\alpha}$ are irreducible, while the main result of
\cite{Pitteloud2001} is that germ-like series of order-value $\omega$
are prime. Moreover, Pommersheim and Shahriari~\cite{Pommersheim2005}
proved that germ-like series of order-value $\omega^2$ have a unique
factorisation, and that some of them are irreducible.

By generalising an argument in~\cite{Berarducci2000}, we shall see
that germ-like series always have factorisations into
irreducibles. Together with Pitteloud's result, we shall be able to
prove that the factorisation into irreducibles of germ-like series of
order-value at most $\omega^3$ must be unique.

\begin{namedthm}{Theorems~\ref{thm:non-crit-fact}-\ref{thm:ufdnc}}
  All non-zero germ-like series in $K((\mathbb{R}^{\leq 0}))$ admit
  factorisations into irreducibles. Every non-zero germ-like series in
  $K((\mathbb{R}^{\leq 0}))$ of order-value $\leq \omega^3$ admits a
  unique factorisation into irreducibles.
\end{namedthm}

For completeness, we shall also verify that irreducible germs and
series of order-value $\omega^3$ do exist.

\begin{namedthm}{Theorems~\ref{thm:irred-germ-w-3}-\ref{thm:irred-series-w-3}}
  There exist irreducible germs in $K((\mathbb{R}^{\leq 0})) / J$ and
  irreducible series in $K((\mathbb{R}^{\leq 0}))$ of order-value
  $\omega^3$.
\end{namedthm}

\subsection*{Further remarks}

As noted before, all the known results about irreducibility and
primality of generalised power series are in fact about germ-like
power series. In view of this, we propose the following conjecture,
which seems to be a reasonable intermediate statement between Conway's
conjecture and Berarducci's conjecture.

\begin{conj}
  Every non-zero germ-like series in $K((\mathbb{R}^{\leq 0}))$ admits
  a unique factorisation into irreducibles.
\end{conj}

In order to treat other series that are not germ-like, we note that
\prettyref{lem:crit-mult} and its following
\prettyref{cor:head-order-value} suggest an alternative multiplicative
order-value map whose value is the \emph{first} term of the Cantor
normal form of the order type, rather than the last infinite one. This
has several consequences about irreducibility of general series; for
instance, if $P$ is the multiplicative group of the non-zero series
with finite support, it implies that the localised ring
$P^{-1}K((\mathbb{R}^{\leq 0}))$ admit factorisation into
irreducibles. Other consequences of the new order-value will be
investigated in a future work.

\section{Preliminaries}

\subsection{Ordinal arithmetic}

This subsection is a self-contained presentation of the classical and
well-known properties of ordinal arithmetic.  First, let us briefly
recall how ordinals can be introduced. Two (linearly) ordered sets $X$
and $Y$ are called order similar if they are isomorphic. The order
similarity is an equivalence relation and its classes are called
\emph{order types}.

An \emph{ordinal number} is the order type of a \emph{well-ordered}
set, i.e., an ordered set with the property that any non-empty subset
has a minimum. Given two ordinal numbers $\alpha, \beta$, we say that
$\alpha \leq \beta$ if there are two representatives $A$ and $B$ such
that $A \subseteq B$ and such that the inclusion of $A$ in $B$ is a
homomorphism; we say that $\alpha < \beta$ if $\alpha \leq \beta$ and
$\alpha \neq \beta$. A key observation in the theory of ordinals is
that $\on$ itself is \emph{well-ordered by $\leq$}. This lets us
define ordinal arithmetic by induction on $\leq$:

\begin{itemize}
\item the minimum ordinal in $\on$ is called \emph{zero} and is
  denoted by $0$;
\item given an $\alpha \in \on$, the \emph{successor} $S(\alpha)$ of
  $\alpha$ is the minimum $\beta$ such that $\beta > \alpha$;
\item given a set $A \subseteq \on$, the \emph{supremum} $\sup(A)$ is
  the minimum $\beta$ such that $\beta \geq \alpha$ for all
  $\alpha \in A$;
\item \emph{sum}: $\alpha + 0 := \alpha$,
  $\alpha + \beta := \sup_{\gamma < \beta} \{S(\alpha + \gamma)\}$;
\item \emph{product}: $\alpha \cdot 0 := 0$,
  $\alpha \cdot 1 := \alpha$,
  $\alpha \cdot \beta := \sup_{\gamma < \beta} \{\alpha \cdot \gamma +
  \alpha\}$;
\item \emph{exponentiation}: $\alpha^0 := 1$,
  $\alpha^\beta := \sup_{\gamma < \beta} \{\alpha^\gamma \cdot
  \alpha\}$.
\end{itemize}

One can easily verify that sum and product are associative, but not
commutative, that the product is distributive over the sum in the
second argument, and that
$\alpha^{\beta + \gamma} = \alpha^\beta \cdot
\alpha^\gamma$. Moreover, sum, product and exponentiation are strictly
increasing and continuous in the second argument.

The \emph{finite} ordinals are the ones that are represented by finite
ordered sets. They can be identified with the natural numbers $0$,
$1$, $2$, $\ldots$, where ordinal arithmetic coincides with Peano's
arithmetic. The ordinals that are not zero or successors are called
\emph{limit}, and one can verify that $\alpha$ is a limit if and only
if $\alpha \neq 0$ and $\alpha = \sup_{\beta < \alpha} \{\beta\}$. The
smallest limit ordinal is called $\omega$.

The three operations admit notions of subtraction, division and
logarithm. More precisely, given $\alpha \leq \beta$, there exist:

\begin{itemize}
\item a unique $\gamma$ such that $\alpha + \gamma = \beta$;
\item unique $\gamma, \delta$ with $\delta < \alpha$ such that
  $\alpha \cdot \gamma + \delta = \beta$;
\item unique $\gamma, \delta, \eta$ with $\delta < \beta$,
  $\eta < \alpha$ such that
  $\beta^{\gamma} \cdot \delta + \eta = \alpha$.
\end{itemize}

In particular, for all $\alpha \in \on$ there is a unique finite
sequence $\beta_1 \geq \beta_2 \geq \ldots \geq \beta_n \geq 0$ such
that
\[
  \alpha = \omega^{\beta_1} + \ldots + \omega^{\beta_n}.
\]
The expression on the right-hand side is called \emph{Cantor normal
  form} of $\alpha$. Given two ordinals in Cantor normal form, it is
rather easy to calculate the Cantor normal form of their sum and
product, using associativity, distributivity in the second argument
and the following rules:
\begin{itemize}
\item if $\alpha < \beta$,
  $\omega^\alpha + \omega^\beta = \omega^\beta$;
\item if $\alpha = \omega^{\beta_1} + \ldots + \omega^{\beta_n}$ is in
  Cantor normal form and $\gamma > 0$, then
  $\alpha \cdot \omega^\gamma = \omega^{\beta_1 + \gamma}$.
\end{itemize}

Finally, we recall that $\on$ also admits different commutative
operations called \emph{Hessenberg's natural sum $\oplus$ and natural
  product} $\odot$. These can be defined rather easily using the
Cantor normal form. Given
$ \alpha = \omega^{\gamma_1} + \omega^{\gamma_2} + \ldots +
\omega^{\gamma_n}$ and
$ \beta = \omega^{\gamma_{n+1}} + \omega^{\gamma_{n+2}} + \ldots +
\omega^{\gamma_{n+m}}$ in Cantor normal form, the natural sum
$\alpha \oplus \beta$ is defined as
$\alpha \oplus \beta := \omega^{\gamma_{\pi(1)}} +
\omega^{\gamma_{\pi(2)}} + \ldots + \omega^{\gamma_{\pi(n+m)}}$, where
$\pi$ is a permutation of the integers $1, \ldots, n+m$ such that
$\gamma_{\pi(1)} \geq \ldots \geq \gamma_{\pi(n+m)}$, and the natural
product is defined by
$\alpha \odot \beta := \bigoplus_{1 \leq i \leq n} \bigoplus_{n+1 \leq
  j \leq k+m} \omega^{\gamma_j \oplus \gamma_j}$.

\subsection{Order-value}

We now recall the definition and the basic properties of the
order-value map introduced by Berarducci in~\cite{Berarducci2000}.

Given $a \in K((\mathbb{R}^{\leq 0}))$, we let $\ot(a)$ be the order
type of support $S_a$ of $a$ (recall that the support of $a$ is a
well-ordered subset of $\mathbb{R}$, hence $\ot(a)$ is a countable
ordinal).  One can verify that given two series
$a,b \in K((\mathbb{R}^{\leq 0}))$ we have:

\begin{itemize}
\item $\ot(a + b) \leq \ot(a) \oplus \ot(b)$;
\item $\ot(a \cdot b) \leq \ot(a) \odot \ot(b)$.
\end{itemize}

However, the above inequalities may well be strict for certain values
of $a$ and $b$. In order to get a better algebraic behaviour,
Berarducci introduced the so called \emph{order-value}
$\vj : K((\mathbb{R}^{\leq 0})) \rightarrow \on$ by considering only
the `tail' of the support.

\begin{defn}
  We let $J$ be the ideal of $K((\mathbb{R}^{\leq 0}))$ generated by
  the set of monomials $ \{ t^\gamma : \gamma \in G^{<0} \}$. For
  every $a \in K((\mathbb{R}^{\leq 0}))$, we call the \emph{germ} of
  $a$ the coset $a+J \in K((\mathbb{R}^{\leq0}))/J$.
\end{defn}

\begin{rem}
  The ideal $J$ can also be defined by looking at the support: for
  every series $a$, $a \in J$ if and only if there exists $\gamma < 0$
  such that $S_a \leq \gamma$. In particular, $a+J = b+J$ if and only
  if there exists $\gamma < 0$ such that for all $\delta \geq \gamma$,
  the coefficients of $t^\delta$ in $a$ and $b$ coincide.
\end{rem}

\begin{nota}
  Let $V \subseteq K((\mathbb{R}^{\leq 0}))$ be any $K$-vector
  space. Then, let us write, for
  $a, \, b \in K((\mathbb{R}^{\leq 0}))$, $a \equiv b \mod V$ if
  $a-b \in V$.
\end{nota}

\begin{defn}
  The \emph{order-value}
  $\vj : K((\mathbb{R}^{\leq 0})) \rightarrow \on$ is defined by:
  \[\vj (a) := \left\{
      \begin{array}{ll}
        0 & \quad \textrm{if } \, a \in J,\\
        1 & \quad \textrm{if } \, a \not\in J \textrm{ and } a \in J + K,\\
        \min \{ \ot(b) : b \equiv  a \, \mod \, J + K \} & \quad \textrm{otherwise}.
      \end{array}
  \right.
  \]
\end{defn}

\begin{rem}\label{rem:vj-cnf}
  Suppose that $a$ is not in $J+K$ and write
  $\ot(a) = \omega^{\alpha_1} + \ldots + \omega^{\alpha_n}$ in Cantor
  normal form. Then $\vj(a)$ is precisely the last infinite term of
  the Cantor normal form of $\ot(a)$ (which is either
  $\omega^{\alpha_{n-1}}$ or $\omega^{\alpha_n}$, depending on whether
  $0 \in S_a$ or $0 \notin S_a$). Note in particular that the
  order-value takes only values of the form $\omega^\alpha$ or $0$.
\end{rem}

Furthermore, since $a \equiv b \mod J$ implies $\vj(a) = \vj(b)$, the
map $\vj$ induces an analogous order-value
$\overline{v}_J : K((\mathbb{R}^{\leq 0})) / J \rightarrow \on$ by
defining $\overline{v}_J(a + J) := \vj(a)$. With a slight abuse of
notation, we will use the symbol $\vj$ for both $\vj$ and
$\overline{v}_J$.

The key and difficult result of~\cite{Berarducci2000} is that the
function $\vj$ is multiplicative.

\begin{prop}[{\cite[Lem.\ 5.5 and Thm.\ 9.7]{Berarducci2000}}]
  Let $a, b \in K((\mathbb{R}^{\leq 0}))$. Then:
  \begin{enumerate}
  \item $\vj(a + b) \leq \, \max\{\vj(a), \vj(b)\}$, with equality if
    $\vj(a) \neq \vj(b)$,
  \item $\vj(ab) = \vj(a) \odot \vj(b)$ \emph{(multiplicative
      property)}.
  \end{enumerate}
\end{prop}

The multiplicative property is the crucial ingredient that leads to
the main results in~\cite{Berarducci2000}. For instance, it implies
the following.

\begin{prop}[{\cite[Thm.\ 9.8]{Berarducci2000}}]
  The ideal $J$ of $ K((\mathbb{R}^{\leq 0}))$ is prime.
\end{prop}
\begin{proof}
  Note that $a \in J$ if and only if $\vj(a)=0$. It follows that for
  all $a, \, b \in K((\mathbb{R}^{\leq 0}))$, if the product $ab$ is
  in $J$, that is, $\vj(ab)=0$, then $\vj(a) \odot \vj(b)=0$, which
  implies $\vj(a)=0$ or $\vj(b)=0$, hence $a \in J$ or $b \in J$.
\end{proof}

Finally, we recall some additional notions and results from
\cite{Berarducci2000}.

\begin{defn}
  Given
  $a = \sum_{\beta} a_\beta t^{\beta} \in K((\mathbb{R}^{\leq 0}))$
  and $\gamma \in \mathbb{R}^{\leq 0}$, we define:
  \begin{itemize}
  \item the \emph{truncation} of $a$ at $\gamma$ is
    $a_{\vert \gamma} := \sum_{\beta \leq \gamma } a_\beta t^{\beta}$,
  \item the \emph{translated truncation} of $a$ at $\gamma$ is
    $a^{\vert \gamma} := t^{-\gamma} a_{\vert \gamma}$.
  \end{itemize}
  The equivalence class $a^{\vert \gamma} + J$ is the \emph{germ of
    $a$ at $\gamma$}.
\end{defn}

It turns out that translated truncations behave like a sort of
`generalised coefficients', as they satisfy the following equation.

\begin{prop}[{\cite[Lem.\ 7.5(2)]{Berarducci2000}}]
  For all $a, b \in K((\mathbb{R}^{\leq 0}))$ and
  $\gamma \in \mathbb{R}^{\leq 0}$ we have:
  \[
    {(ab)}^{\vert\gamma} \equiv \sum_{\delta + \varepsilon = \gamma}
    a^{\vert\delta} b^{\vert\varepsilon} \mod J \quad
    \textrm{\emph{(convolution formula)}}.
  \]
\end{prop}

\section{Primality in $K((\mathbb{R}^{\leq0})) / J$}

Pitteloud~\cite{Pitteloud2001} proved that the series of order type
$\omega$ or $\omega+1$ and order-value $\omega$ are \emph{prime}. In
this section, we adopt the same strategy to prove that every germ
$a \in K((\mathbb{R}^{\leq 0})) / J$ of order-value $\omega$ is prime.

Following~\cite[p.~1209]{Pitteloud2001}, we introduce some additional
$K$-vector spaces.

\begin{defn}[{\cite[p.~1209]{Pitteloud2001}}]
  For $\alpha \in \on$, let $J_{\omega^{\alpha}}$ be the $K$-vector
  space
  $J_{\omega^{\alpha}} := \{ a \in K((\mathbb{R}^{\leq 0})) \,:\,
  \vj(a)<\omega^{\alpha} \}$. Moreover, we write
  $b \vert a \mod J_{\omega^\alpha}$ if there exists
  $c \in K((\mathbb{R}^{\leq 0}))$ such that
  $a \equiv bc \mod J_{\omega^\alpha}$.
\end{defn}

For instance, $J_{\omega^{0}} = J$ and $J_{\omega^1} = J + K$. Note
that $J_{\omega^\alpha}$ is just a $K$-vector space. By the
multiplicative property, one can easily verify that
$J_{\omega^\alpha}$ is closed under multiplication if and only if
$\alpha = \omega^\beta$ for some $\beta$, and it is an ideal if and
only if $\alpha = 0$.

Let $a, b, c, d \in K((\mathbb{R}^{\leq 0}))$ satisfy $ab = cd$ and
$\vj(a) = \ot(a) = \omega$.  Pitteloud proved that either $a \vert c$
or $a \vert d$ in $K((\mathbb{R}^{\leq 0}))$ (where $a \vert c$ means
$a$ divides $c$) by analysing the related equation $a^k b = c^l d$
(with $k,l > 0$). More precisely, he proved the following:

\begin{prop}[{\cite[Prop.\ 3.2]{Pitteloud2001}}]\label{prop:pitteloud}
  Let $a$, $b$, $c$, $d$ in $K((\mathbb{R}^{\leq 0}))$ be such that
  $\vj(a) = \omega$ and assume that
  $a^k b = c^l d \mod J_{\vj(a^{k}b)}$ with $k,l > 0$. Then either
  $a \vert c \mod J_{\vj(c)}$ or $a \vert d \mod J_{\vj(d)}$.
\end{prop}

Starting from this proposition, we can prove the following result.

\begin{thm}\label{thm:prime-in-j}
  All germs in $K((\mathbb{R}^{\leq 0})) / J$ of order-value $\omega$
  are prime.
\end{thm}
\begin{proof}
  Let $A = a + J$, $B = b + J$, $C = c + J$ and $D = d + J$ be
  non-zero germs of $K((\mathbb{R}^{\leq 0})) / J$ such that
  $\vj(A) = \omega$ and $AB = CD$.  We claim that $A \vert C$ or
  $A \vert D$. We work by induction on $\vj(AB)$.

  Note that $\vj(ab-cd) = \vj(AB-CD)=0$, so there exists $j \in J$
  such that $ab = cd + j$. Since
  $\vj(j) = 0 < \omega = \vj(a) \leq \vj(ab)$, we have in fact
  $ab = cd \mod J_{\vj(ab)}$. By \prettyref{prop:pitteloud}, either
  $a \vert c \mod J_{\vj(c)}$ or $a \vert d \mod J_{\vj(d)}$.

  Since $C$ and $D$ have a symmetric role, we may assume to be in the
  former case. Then there are $e,c' \in K((\mathbb{R}^{\leq 0}))$ such
  that $c = ae + c'$ with $\vj(c') < \vj(c)$. In turn, we have
  \[
    ab = cd + j = (ae + c')d + j
  \]
  and in particular
  \[
    a(b - ed) = c'd + j.
  \]
  Let $B' := (b-ed) + J$, $C' := c' + J$, $ E := e + J$. The above
  equation means that $AB' = C'D$. Now note that
  $\vj(AB') = \vj(C'D) < \vj(CD) = \vj(AB)$. Therefore, by inductive
  hypothesis, either $A \vert C'$ or $A \vert D$. In the latter case,
  we are done. In the former case, we just recall that $C = AE + C'$,
  so $A \vert C$, as desired.
\end{proof}

\begin{thm}\label{thm:ufdq}
  Every non-zero germ in $K((\mathbb{R}^{\leq 0})) / J$ of order-value
  $\leq \omega^3$ admits a unique factorisation into irreducibles.
\end{thm}
\begin{proof}
  If $\vj(A) = \omega$, then $A$ is prime and therefore irreducible by
  \prettyref{thm:prime-in-j}.

  If $\vj(A) = \omega^2$ or $\omega^3$, then $A$ is either irreducible
  or equal to $A = BC$ with $\vj(B), \vj(C) <\vj(A)$. We assume to be
  in the latter case.  Since $\vj(B) \odot \vj(C) = \vj(A)$, we must
  have that either $\vj(B) = \omega$ or $\vj(C) = \omega$; by
  symmetry, we may assume that $\vj(B) = \omega$, and in particular
  that $B$ is prime.

  If $A$ has another factorisation into irreducibles, then $B$ must
  divide one of the factors, and in particular it must be equal to one
  of the factors up to a unit. The product of the remaining factors
  has either order-value $\omega^2$ or $\omega$; in both cases we
  repeat the argument and we are done.
\end{proof}

\section{Germ-like series}

Unfortunately, even if a series in ${K((\mathbb{R}))}^{\leq 0}$ has an
irreducible germ, it may well be reducible (see for instance
$(t^{-1} -1) (1 + \sum_{n} t^{-\frac{1}{n}})$). This implies that the
results on germs cannot be lifted automatically to all series. On the
other hand, there are some series which behave similarly enough to
germs so that the same techniques can be applied to them.

\begin{defn}\label{def:germ-like}
  We say that an $a \in K((\mathbb{R}^{\leq 0}))$ is \emph{germ-like}
  if either $\ot(a) = \vj(a)$ or $\vj(a) > 1$ and
  $\ot(a) = \vj(a) + 1$.
\end{defn}

We shall prove that if a product of non-zero series is germ-like, then
the series themselves are germ-like. For this, we recall the following
definition from~\cite{Berarducci2000}.

\begin{defn}
  Given $a \in K((\mathbb{R}^{\leq 0}))$, let
  $\alpha = \max \{\vj(a^{\vert\gamma}) : \gamma \in \mathbb{R}^{\leq
    0}\}$. The \emph{critical point} $\crit(a)$ of $a$ is the least
  $\gamma \in \mathbb{R}^{\leq 0}$ such that
  $\vj(a^{\vert\gamma}) = \alpha$.
\end{defn}

\begin{rem}\label{rem:vj-crit}
  By construction, $\vj(a^{\vert\crit(a)})$ is equal to the first term
  of the Cantor normal form of the order type of $a$.
\end{rem}

Proving that a critical point always exists is not difficult, and we
refer to~\cite[\S 10]{Berarducci2000} for the relevant details.

\begin{lem}\label{lem:germ-like-crit}
  If $a \in K((\mathbb{R}^{\leq 0}))$, $a$ is germ-like if and only if
  $\crit(a) = 0$.
\end{lem}
\begin{proof}
  First of all, we note that if $a \in J$, then $\vj(a) = 0$; in this
  case, we just note that $\ot(a) = 0$ if and only if $a = 0$, and the
  conclusion follows trivially. Similarly, if $a \in J + K$ but
  $a \notin J$, then $\vj(a) = 1$, and clearly $\ot(a) = 1$ if and
  only if $a \in K^*$, proving again the conclusion.

  Now assume that $a \notin J + K$. Recall that in this case $\vj(a)$
  is the last infinite term of the Cantor normal form of $\ot(a)$ (in
  particular, $\vj(a) > 1$; see \prettyref{rem:vj-cnf}). Therefore,
  $\ot(a) = \vj(a)$ holds if and only if the Cantor normal form of
  $\ot(a)$ is $\omega^\alpha$ for some $\alpha \in \on$. Similarly,
  $\vj(a) = \ot(a) + 1$ holds if and only if the Cantor normal form of
  $\ot(a)$ is $\omega^{\alpha} + 1$ for some $\alpha \in \on$.

  If $\ot(a)$ is $\omega^\alpha$ or $\omega^\alpha + 1$, since
  $a \notin J+K$, we have $\ot(a^{\vert\gamma}) < \omega^{\alpha}$ for
  all $\gamma \in \mathbb{R}^{<0}$. In particular,
  $\vj(a^{\vert\gamma}) \leq \ot(a^{\vert\gamma}) < \omega^{\alpha} =
  \vj(a)$ for all $x\in \mathbb{R}^{<0}$, hence $\crit(a) = 0$.

  Otherwise, we have that the Cantor normal form of $\ot(a)$ is
  $\omega^{\beta_1} + \ldots + \omega^{\beta_k}$ or
  $\omega^{\beta_1} + \ldots + \omega^{\beta_k} + 1$, where $k > 1$
  and $\beta_k > 0$; let $\gamma \in \mathbb{R}^{\leq0}$ be the
  minimum real number such that
  $\ot(a^{\vert\gamma}) \geq \omega^{\beta_1}$. Then $\gamma$ is
  negative, and it is easy to verify that
  $\vj(a^{\vert\gamma}) = \omega^{\beta_1}$ and in fact that
  $\gamma = \crit(a)$, thus $ \crit(a) < 0$.
\end{proof}

The following lemma is inspired by~\cite[\S 10]{Berarducci2000}.

\begin{lem}\label{lem:crit-mult}
  If $b, c$ are non-zero series of $K((\mathbb{R}^{\leq 0}))$, then
  \begin{enumerate}
  \item $\crit(bc) = \crit(b) + \crit(c)$;
  \item
    $\vj({(bc)}^{\vert\crit(bc)}) = \vj(b^{\vert\crit(b)}) \odot
    \vj(c^{\vert\crit(c)})$.
  \end{enumerate}
\end{lem}
\begin{proof}
  We proceed as in~\cite[Lemma 10.4]{Berarducci2000}. Let
  $\gamma = \crit(b)$ and $\delta = \crit(c)$. By the convolution
  formula, for any $\varepsilon \in \mathbb{R}^{\leq 0}$ we have
  \[
    {(bc)}^{\vert\varepsilon} \equiv \sum_{\gamma' + \delta' =
      \varepsilon} b^{\vert\gamma'}c^{\vert\delta'} \mod J.
  \]
  So,
  $\vj({(bc)}^{\vert\varepsilon}) \leq \mbox{max}_{\gamma' + \delta' =
    \varepsilon} \{\vj( b^{\vert\gamma'}) \odot \vj(c^{\vert\delta'})
  \}$. It follows at once that
  $\vj({(bc)}^{\vert\varepsilon}) \leq \vj(b^{\vert\gamma}) \odot
  \vj(c^{\vert\delta})$, and if the equality is attained, then for
  some $\gamma'$, $\delta'$ such that
  $\gamma' + \delta' = \varepsilon$ we have
  $\vj(b^{\vert\gamma'}) = \vj(b^{\vert\gamma})$ and
  $\vj(c^{\vert\delta'}) = \vj(c^{\vert\delta})$. In particular, if
  the equality holds, then
  $\varepsilon = \gamma' + \delta' \geq \gamma + \delta$.

  On the other hand,
  \[
    {(bc)}^{\vert\gamma + \delta} \equiv
    b^{\vert\gamma}c^{\vert\delta} +
    \sum_{\substack{\gamma' + \delta' = \gamma + \delta \\
        \gamma' < \gamma}} b^{\vert\gamma'}c^{\vert\delta'} +
    \sum_{\substack{\gamma' + \delta' = \gamma + \delta \\
        \delta' < \delta}} b^{\vert\gamma'}c^{\vert\delta'} \mod J.
  \]
  It immediately follows that
  $\vj({(bc)}^{\vert\gamma + \delta}) =
  \vj(b^{\vert\gamma}c^{\vert\delta}) = \vj(b^{\vert\gamma}) \odot
  \vj(c^{\vert\delta})$. Therefore, $\crit(bc) = \gamma + \delta$,
  proving both conclusions.
\end{proof}

\begin{cor}\label{cor:germ-like-prod}
  Let $a,b,c \in K((\mathbb{R}^{\leq 0}))$ be non-zero series with
  $a = bc$. Then $a$ is germ-like if and only if $b$ and $c$ are
  germ-like.
\end{cor}
\begin{proof}
  By Lemmas~\ref{lem:germ-like-crit},~\ref{lem:crit-mult}, it suffices
  to note that $\crit(a) = \crit(b) + \crit(c)$, so $\crit(a) = 0$ if
  and only if $\crit(b) = \crit(c) = 0$.
\end{proof}

\begin{cor}\label{cor:head-order-value}
  The function $a \mapsto \vj(a^{\vert\crit(a)})$ is multiplicative.
\end{cor}

\begin{thm}\label{thm:non-crit-fact}
  All non-zero germ-like series in $K((\mathbb{R}^{\leq 0}))$ admit
  factorisations into irreducibles.
\end{thm}
\begin{proof}
  We prove the conclusion by induction on $\vj(a)$. If $\vj(a) = 1$,
  then we must have $\ot(a) = \vj(a) = 1$, which implies that
  $a \in K$, and the conclusion follows trivially. If $\vj(a) > 1$,
  then either $a$ is irreducible, in which case we are done, or
  $a = bc$ for some $b, c \in K((\mathbb{R}^{\leq 0})) \setminus
  K$. By \prettyref{cor:germ-like-prod}, $b$ and $c$ are germ-like. If
  $\vj(b) = 1$, then by the previous argument we have $b \in K$, a
  contradiction, hence $\vj(b) > 1$; similarly, we deduce that
  $\vj(c) > 1$ as well. By the multiplicative property, it follows
  that $\vj(b), \vj(c) < \vj(a)$; by induction, $b$ and $c$ have a
  factorisation into irreducibles, and we are done.
\end{proof}

\begin{thm}\label{thm:ufdnc}
  Every non-zero germ-like series in $K((\mathbb{R}^{\leq 0}))$ of
  order-value $\leq \omega^3$ admits a unique factorisation into
  irreducibles.
\end{thm}
\begin{proof}
  Let $a \in K((\mathbb{R}^{\leq 0}))$ be a germ-like series of order
  value $\vj(a) \leq \omega^3$. Suppose that $a$ has a non-trivial
  factorisation $a = bc$. By \prettyref{cor:germ-like-prod}, $b$ and
  $c$ are germ-like. By the multiplicative property, and possibly by
  swapping $b$ and $c$, we may assume $\vj(b) \leq \omega$.

  If $\vj(b) = 1$, then $\ot(b) = 1$, hence $b \in K$, contradicting
  the hypothesis that the factorisation is non-trivial. It follows
  that $\vj(b) = \omega$. Since $b$ is germ-like, we have
  $\ot(b) = \omega$ or $\ot(b) = \omega + 1$, in which case $b$ is
  prime by~\cite[Thm.\ 3.3]{Pitteloud2001}. Once we have a prime
  factor, one can deduce easily that the factorisation is unique, as
  in the proof of \prettyref{thm:ufdq}.
\end{proof}

Moreover, we observe that irreducible germs do lift to irreducible
germ-like series.

\begin{prop}\label{prop:germ-like-irred}
  Let $a \in K((\mathbb{R}^{\leq 0}))$ be germ-like. If the germ
  $a + J$ of $a$ is irreducible in $K((\mathbb{R}^{\leq 0})) / J$,
  then $a$ is irreducible in $K((\mathbb{R}^{\leq 0}))$.
\end{prop}
\begin{proof}
  Suppose that $a = bc$ and that $a + J$ is irreducible. It
  immediately follows that one of $b + J$ or $c+J$ is a unit in
  $K((\mathbb{R}^{\leq 0}))$, say $b + J$. By the multiplicative
  property, we have that $\vj(b) = 1$. Since $b$ is germ-like it
  follows that $\ot(b) = 1$, which implies that $b \in K^*$, hence
  that $b$ is a unit in $K((\mathbb{R}^{\leq 0}))$, as desired.
\end{proof}

\begin{rem}
  In general, the converse does not hold. Indeed, let $a, b$ be two
  germ-like series of order-value $\omega$. By~\cite[Cor.\
  3.4]{Pommersheim2005}, for all $\gamma \in \mathbb{R}^{<0}$ except
  at most countably many, $ab + t^{\gamma}$ is irreducible, while of
  course its germ $ab + t^{\gamma} + J = ab + J = (a + J)(b + J)$ is
  reducible.
\end{rem}

\section{Irreducibles of order-value $\omega^3$}\label{val3}

We conclude by showing that there are several irreducible elements in
both $K((\mathbb{R}^{\leq 0}))$ and $K((\mathbb{R}^{\leq 0})) / J$ of
order-value $\omega^3$.  We follow a strategy similar to the one of
\cite{Pommersheim2005}.

\begin{lem}\label{lem:small-gamma}
  Let $a, b \in K((\mathbb{R}^{\leq 0}))$. If
  $a \equiv b \mod J_{\omega^{\alpha + 1}}$, then for all
  $\gamma \in \mathbb{R}^{< 0}$ sufficiently close to $0$ we have
  $a^{\vert\gamma} \equiv b^{\vert\gamma} \mod J_{\omega^\alpha}$.
\end{lem}
\begin{proof}
  Let $c = a-b$, so that $\vj(c) < \omega^{\alpha + 1}$. Note that for
  all $\gamma < 0$ sufficiently close to $0$ we have
  $\vj(c^{\vert\gamma}) < \vj(c)$, because then the tail of
  $c^{\vert\gamma}$ is a proper (translated) truncation of the tail of
  $c$. In particular,
  $\vj(c^{\vert\gamma}) < \omega^{\alpha} = \vj(c)$. Since
  $a^{\vert\gamma} - b^{\vert\gamma} = c^{\vert\gamma}$, we then have
  $a^{\vert\gamma} \equiv b^{\vert\gamma} \mod J_{\omega^\alpha}$.
\end{proof}

Given $a \in K((\mathbb{R}^{\leq 0}))$ and
$\gamma \in \mathbb{R}^{< 0}$, we let $V_\gamma(a)$ be the $K$-linear
space generated by all the germs of $a$ between $\gamma$ and $0$,
modulo $J + K$; formally,
\[
  V_\gamma(a) := \mathrm{span}_K \{a^{\vert\delta} + J + K \,:\,
  \gamma < \delta < 0 \} \subseteq K((\mathbb{R}^{\leq 0})) / (J + K).
\]
In particular, $V_{\gamma}(a)$ contains all the germs
$a^{\vert\delta} + J$, for $\gamma < \delta < 0$, modulo the subspace
$J + K$.

Note that the spaces $V_{\gamma}(a)$ form a directed system, as
clearly $V_{\gamma}(a) \supseteq V_{\gamma'}(a)$ for
$\gamma < \gamma'$. We let $V(a)$ be their intersection for
$\gamma \in \mathbb{R}^{<0}$:
\[
  V(a) := \bigcap_{\gamma<0} V_\gamma(a).
\]
The spaces $V(a)$ contains, for instance, the germs $b + J$ (modulo
$J + K$) that appear repeatedly as $b + J = a^{\vert\delta} + J$ for
$\delta$ approaching $0$.

\begin{rem}\label{rem:V-germ-invariant}
  If $a \equiv b \mod J$, then $V_\gamma(a) = V_\gamma(b)$ for all
  $\gamma < 0$ sufficiently close to $0$, so in fact $V(a) = V(b)$. In
  particular, it is well defined to write $V(a + J) := V(a)$ for any
  given germ $a + J$.
\end{rem}

\begin{prop}\label{prop:dim-V-leq-2}
  Let $a, b, c \in K((\mathbb{R}^{\leq 0}))$ be such that
  $a \equiv bc \mod J_{\omega^2}$ and $\vj(b) = \vj(c) = \omega$. Then
  $\dim(V(a)) \leq 2$.
\end{prop}
\begin{proof}
  By \prettyref{lem:small-gamma}, for all $\gamma < 0$ sufficiently
  close to $0$ we have
  $a^{\vert\gamma} \equiv {(bc)}^{\vert\gamma} \mod J_\omega = J +
  K$. By the convolution formula,
  ${(bc)}^{\vert\gamma} \equiv \sum_{\delta + \varepsilon = \gamma}
  b^{\vert\delta} c^{\vert\varepsilon} \mod J$.

  Note that when $\delta < 0$ is sufficiently close to $0$, we have
  $b^{\vert\delta}, c^{\vert\delta} \in J + K$. This in particular
  implies that if $\gamma$ is sufficiently close to $0$, then
  ${(bc)}^{\vert\gamma} \equiv b^{\vert\gamma}c + bc^{\vert\gamma}
  \mod J_{\omega} = J + K$. Moreover, we must have
  $b^{\vert\gamma}, c^{\vert\gamma} \in J + K$; in other words,
  $b^{\vert\gamma} = k_{b} + j_{b}$, $c^{\vert\gamma} = k_{c} + j_{c}$
  for some $k_{b}, k_{c} \in K$ and $j_{b}, j_{c} \in J$. It follows
  that when $\gamma < 0$ is sufficiently close to 0, we have
  \[
    a^{\gamma} \equiv k_{b}c + k_{c}b + j_{b}c + j_{c}b \equiv k_{b}c
    + k_{c}b \mod J + K.
  \]
  Therefore, $V(a)$ is generated as $K$-vector space by $b + J + K$
  and $c + J + K$, hence $\dim(V(a)) \leq 2$.
\end{proof}

In order to find a sufficient criterion for irreducibility of series
of order-value $\omega^3$, we picture a series
$a \in K((\mathbb{R}^{\leq 0}))$ of order-value $\omega^{\alpha + 1}$
as if it were a series of order-value $\omega$ with coefficients that
are themselves series of order-value $\omega^\alpha$. In other words,
we describe $a$ as the sum of $\omega$ series of order-value
$\omega^\alpha$.

\begin{defn}
  Let $\alpha \in \on$ and $a \in K((\mathbb{R}^{\leq 0}))$ be such
  that $\vj(a) = \omega^{\alpha + 1}$. We say that $\gamma \in S_a$ is
  a \emph{big point} of $a$ if $\vj(a^{\vert\gamma}) = \omega^\alpha$.
\end{defn}

\begin{rem}
  By construction, the big points of a series must accumulate to $0$.
\end{rem}

We can use big points to give the following sufficient criterion for
the irreducibility of a series in $K((\mathbb{R}^{\leq 0}))$ of
order-value $\omega^3$.

\begin{prop}\label{prop:dim-V-r-s-leq-2}
  Let $a, b, c \in K((\mathbb{R}^{\leq 0}))$ be such that
  $a \equiv bc \mod J$, $\vj(b) = \omega$ and $\vj(c) = \omega^2$. Let
  $\gamma, \delta$ be two big points of $a$ sufficiently close to
  $0$. Then there exist $r, s \in K$, not both zero, such that
  $\dim(V(ra^{\vert\gamma} + sa^{\vert\delta})) \leq 2$.
\end{prop}
\begin{proof}
  By the convolution formula, when $\gamma, \delta$ are sufficiently
  close to $0$ we have
  \[
    a^{\vert\gamma} \equiv b^{\vert\gamma}c + bc^{\vert\gamma} \mod
    J_{\omega^2},\quad a^{\vert\delta} \equiv b^{\vert\delta}c +
    bc^{\vert\delta} \mod J_{\omega^2}.
  \]
  If $b^{\vert\gamma} = b^{\vert\delta} = 0$, then we can take
  $r = 1, s = 0$ and apply \prettyref{prop:dim-V-leq-2} to obtain the
  conclusion. Otherwise, we have
  \[
    b^{\vert\delta}a^{\vert\gamma} - b^{\vert\gamma}a^{\vert\delta}
    \equiv b(b^{\vert\delta}c^{\vert\gamma} -
    b^{\vert\gamma}c^{\vert\delta}) \mod J_{\omega^2}.
  \]
  Let $r := b^{\vert\delta}$, $s := -b^{\vert\gamma}$.  By the
  previous equation and \prettyref{prop:dim-V-leq-2} we get
  $V(ra^{\vert\gamma} + sa^{\vert\delta}) \leq 2$, as desired.
\end{proof}

It is now not difficult to construct several series
$a \in K((\mathbb{R}^{\leq 0}))$ of order-value $\omega^3$ such that
the condition
\[
  \dim(V(ra^{\vert\gamma}+sa^{\vert\delta})) > 2
\]
is satisfied for all $r, s \in K$ not both zero and for all distinct
big points $\gamma \neq \delta$ of $a$. In particular, their
corresponding germs $a + J$ are all irreducible. Indeed, we observe
the following.

\begin{lem}\label{lem:construct-series}
  Let $(a_i \in K((\mathbb{R}^{\leq 0})) \,:\, i \in \mathbb{N})$ be a
  sequence of series of order type $\omega^\alpha$, and let
  $(\gamma_i \in \mathbb{R}^{<0} \,:\, i \in \mathbb{N})$ be a
  strictly increasing sequence of negative real numbers such that
  $\lim_{i \to \infty} \gamma_{i} = 0$. Then there exists
  $a \in K((\mathbb{R}^{\leq 0}))$ with $\vj(a) = \omega^{\alpha+1}$
  whose big points are the exponents $\gamma_{i}$ and such that
  $a^{\vert\gamma_i} \equiv a_i \mod J$ for all $i \in \mathbb{N}$.
\end{lem}
\begin{proof}
  First of all, we may assume that
  $S_{a_{i + 1}} > \gamma_i - \gamma_{i + 1}$ for all
  $i \in \mathbb{N}$. Indeed, if this is not the case, it suffices to
  replace $a_{i + 1}$ with the series
  $a_{i + 1} - {(a_{i + 1})}_{\vert\gamma_i - \gamma_{i + 1}} \equiv a_{i
    + 1} \mod J$. In particular, by construction, the support of the
  series $t^{\gamma_{i}}a_{i}$ do not overlap. This implies that the
  following infinite sum is well defined:
  \[
    a := \sum_{i \in \mathbb{N}} t^{\gamma_i}a_i.
  \]

  Since $S_{t^{\gamma_{i + 1}}a_{i + 1}} > \gamma_i$, we then have
  \[
    a^{\vert\gamma_i} = \sum_{j \leq i} t^{\gamma_j - \gamma_i} a_j \equiv
    a_i \mod J,
  \]
  and the exponents $\gamma_{i}$ are exactly the big points of $a$.
\end{proof}

\begin{thm}\label{thm:irred-germ-w-3}
  There exist irreducible germs in $K((\mathbb{R}^{\leq 0})) / J$ of
  order-value $\omega^3$.
\end{thm}
\begin{proof}
  Let $\Omega$ be any countable set of series of order-value $\omega$
  with pairwise disjoint supports. Clearly, $\Omega$ is a $K$-linearly
  independent set. Moreover, it is $K$-linearly independent even
  modulo the vector space $J+K$.  Let $a_{i}$, for $i \in \mathbb{N}$,
  be some enumeration of $\Omega$, and take a strictly increasing
  sequence $(\gamma_i \in \mathbb{R}^{< 0} \,:\, i \in\mathbb{N})$
  such that $\lim_{i \to \infty} \gamma_i = 0$.

  By \prettyref{lem:construct-series}, there are series $b_i$ of
  order-value $\omega^2$ such that
  \[
    b_i^{\vert\gamma_{3j + k}} \equiv a_{3i + k} \mod J.
  \]
  for all $i, j \in \mathbb{N}$ and $k \in \{0, 1, 2\}$. By
  \prettyref{lem:construct-series} again, there exists $c$ of
  order-value $\omega^3$ such that
  \[
    c^{\vert\gamma_i} \equiv b_i \mod J
  \]
  and whose big points are the $\gamma_{i}$'s.

  Take any $r, s \in K$ not both zero and any two distinct
  $\gamma_{i}, \gamma_{j}$. By construction,
  $rc^{\vert\gamma_{i}} + sc^{\vert\gamma_{j}} = rb_{i} +
  sb_{j}$. Note moreover that for any $\gamma_{3h+k}$,
  \[
    (rb_{i} + sb_{j})^{\vert\gamma_{3h+k}} =
    rb_{i}^{\vert\gamma_{3h+k}} + sb_{j}^{\vert\gamma_{3h+k}} \equiv
    ra_{3i+k} + sa_{3j+k} \mod J.
  \]
  It follows at once that
  \begin{multline*}
    V(rc^{\vert\gamma_i} + sc^{\vert\gamma_j}) = V(rb_{i} + sb_{j}) \supseteq \\
    \{ra_{3i} + sa_{3j} + J + K, ra_{3i + 1} + sa_{3j + 1} + J + K,
    ra_{3i + 2} + sa_{3j + 2} + J + K\}.
  \end{multline*}
  By elementary linear algebra, it follows that for all $i \neq j$ we
  have
  \[
    \dim(V(rc^{\vert\gamma_i} + sc^{\vert\gamma_j})) \geq 3.
  \]
  By \prettyref{prop:dim-V-r-s-leq-2}, it follows that $c + J$ is
  irreducible.
\end{proof}

\begin{thm}\label{thm:irred-series-w-3}
  There exist irreducible series in $K((\mathbb{R}^{\leq 0}))$ of
  order-value $\omega^3$.
\end{thm}
\begin{proof}
  By \prettyref{thm:irred-germ-w-3}, there exist series $c$ of
  order-value $\omega^3$ such that $c + J$ is irreducible. Up to
  replacing $c$ with $c - c_{\vert\gamma}$ for a suitable
  $\gamma \in \mathbb{R}^{< 0}$, we may directly assume that $c$ is
  germ-like. By \prettyref{prop:germ-like-irred}, $c$ is irreducible.
\end{proof}

\bibliographystyle{plainurl}

\end{document}